\documentclass{article}
\usepackage{amsmath}    
\usepackage{longtable}  
\usepackage{latexsym}	
\usepackage{amssymb}	
\usepackage{epsfig}
\usepackage{algorithmic}

\newtheorem{algorithm}{Algorithm}[section]

\newtheorem{theorem}{Theorem}[section]

\newtheorem{remark}{Remark}[section]

\newcommand{\qed}{\nobreak \ifvmode \relax \else \ifdim\lastskip<1.5em \hskip-\lastskip \hskip1.5em plus0em minus0.5em \fi \nobreak \vrule height0.75em width0.5em depth0.25em\fi} 

\newenvironment{proof}[1][Proof:]{\begin{trivlist} 
\item[\hskip \labelsep {\bfseries #1}]}{\end{trivlist}}

\def\R{{\bf R}}

\def\T{{\rm T}}

\begin{document}

\title{A robust BFGS algorithm for unconstrained nonlinear optimization problems}

\author{Yaguang Yang\thanks{
NRC, Office of Research, 21 Church Street, Rockville, 20850. Email: 
yaguang.yang@verizon.net}}

\date{\today}

\maketitle    

\begin{abstract}
The traditional BFGS algorithm has been proved very 
efficient. It is convergent for convex nonlinear optimization 
problems. However, for non-convex nonlinear optimization 
problems, it is known that the BFGS algorithm may not
be convergent. This paper proposes a robust BFGS algorithm
in the sense that the algorithm superlinearly converges to
a local minimum under some mild assumptions for
both convex and non-convex nonlinear optimization problems.
Numerical test on the CUTEst test set is reported to demonstrate
the merit of the proposed robust BFGS algorithm. 
This result shows that the robust 
BFGS algorithm is very efficient and effective.
\end{abstract}

{\bf keywords:} robust BFGS algorithm, Global convergence, Superlinear convergence,
unconstrained optimization.


\newpage
\section{ Introduction}
The robust optimization considers optimization problems with 
uncertainty, in which the uncertainty model is not stochastic. 
The decision maker constructs a solution that is feasible to any 
realization of the uncertainty in a given set. Although, this idea
became popular in the later of 1990s (see 
\cite{bggn03,bmn00,glp13,goo03,hladik16}),
similar idea was proposed by many researchers over last 70 years, 
for example, \cite{beale55,dantzig55, yang85,yang91}. 
Currently, many robust optimization algorithms consider 
a set of parameters in the optimization model that are 
fixed in traditional optimization problems. Some widely 
cited review papers are available in \cite{bbc11,rm19}. 
In this paper, we consider a different robust optimization 
algorithm that converges globally and superlinearly to a local minimum under
some mild assumptions for both convex and non-convex
optimization problems. In particular, we propose a robust 
BFGS algorithm, which is insensitive to the mathematical models,
convex or non-convex.

The BFGS algorithm is one of the most successful algorithms for 
unconstrained nonlinear programming \cite{bertse96}. Although 
global and superlinear convergence results have been established 
for convex problems \cite{powell76}, it has been proved that, 
the BFGS algorithm with Wolfe line search may not be convergent 
for general non-convex nonlinear functions \cite{dai02}. 
Unfortunately, Wolfe line search condition is one of the prerequisites 
for applying the Zoutendijk theorem \cite{zoutendijk70} 
to prove the global convergence of optimization algorithms. This 
motivates us to find a robust BFGS algorithm that is globally 
convergent for all twice differentiable nonlinear functions, convex 
or non-convex. We also would like the behavior of the robust BFGS 
algorithm to be the same as the behavior of BFGS algorithm
for convex problems, when the iterates 
approach a minimizer where the strong second order condition is met,
 i.e., we would like the proposed algorithm to be locally superlinearly 
convergent for non-convex optimization problems.

We will first examine how a robust Newton algorithm \cite{yyang13}
achieves global and quadratic convergence so that we can device a
robust BFGS algorithm to achieve global and superlinear convergence. 
The robust Newton algorithm in \cite{yyang13} uses a robust Hessian 
matrix which is a convex combination of the Hessian for Newton's 
algorithm and the identity matrix for the steepest descent algorithm. 
The most obvious advantage of using the convex combination other 
than a linear combination of these matrices is that the robust Newton 
algorithm may take the steepest descent iteration or Newton's 
iteration; it has the merits of both the steepest descent algorithm 
and Newton's algorithm, i.e., it is globally and quadratically convergent.
Similar to the idea that the BFGS estimates the Hessian matrix, we 
propose a robust BFGS update that estimates the robust Hessian 
matrix given in \cite{yyang13}. 
The robust BFGS matrix estimates a robust Hessian matrix 
which is a convex combination of an identity matrix for the 
steepest descent algorithm and a Hessian matrix for Newton's 
algorithm. The coefficient of the convex combination in the 
robust BFGS algorithm is dynamically chosen in every iteration. 
It is proved that, for any twice differentiable nonlinear function 
(convex or non-convex), the algorithm is globally convergent to 
a stationary point. If the stationary point is a local minimizer 
where the Hessian is strongly positive definite in a neighborhood 
of the minimizer, the iterates will eventually enter and stay in the 
neighborhood, and the robust BFGS algorithm reduces to the 
BFGS algorithm in this neighborhood. Therefore, the robust 
BFGS algorithm is superlinearly convergent. Moreover, the 
computational cost of the robust BFGS in each iteration is 
almost the same as the cost of the BFGS. 


The robust BFGS is implemented in the Matlab function {\tt mBFGS}. 
The implementation {\tt mBFGS} and an implementation of BFGS in 
Matlab Optimization Toolbox {\tt fminunc}
are tested against the CUTEst test set \cite{web}.
The performance of {\tt mBFGS} is compared to {\tt fminunc},
and other established and/or state-of-the-art optimization software, 
such as a limited memory BFGS algorithm \cite{nocedal80} implemented 
as {\tt L-BFGS}, a descent conjugate gradient algorithm \cite{hz05} 
implemented as {\tt CG-Descent 5.3}, and a limited memory descent 
and conjugate algorithm \cite{hz12} implemented as {\tt L-CG-Descent}. 
This result shows that the robust BFGS is very efficient and effective.

The remainder of the paper is organized as follows. 
Section 2 introduces the robust BFGS algorithm. Section 3 
discusses the algorithm's convergence properties. Section 4 
provides the test results. Section 5 summarizes the conclusions.

\section{The robust BFGS Method}

Our objective is to robustly minimize a multi-variable nonlinear 
(convex or nonconvex) function
\begin{equation}
\min f(x),
\label{obj}
\end{equation}
where $f$ is twice differentiable and $x \in \R^{n}$. 
Throughout the paper, we define by $g(x)$ or simply $g$ the 
gradient of $f(x)$, by $H(x)$ or simply $H$ the Hessian of 
$f(x)$. We denote $H \succ 0$ if a matrix $H$ is positive definite,
$H \succeq 0$ if $H$ is positive semi-definite. We denote by 
$\bar{x}$ a local minimizer of (\ref{obj}) that satisfies
\begin{equation}
g(\bar{x})=0.
\end{equation}
We will use subscript $k$ for the $k$th iteration,
hence, $x_0$ is used to represent the initial point. 
We make the following assumptions in the convergence analysis. 

\medskip
{\bf Assumptions:}
\begin{itemize}
\item[1.] For an open set ${\cal M}$ containing the level set 
${\cal L}=\{ x: f(x) \le f(x_0) \}$, $g(x)$ is Lipschitz continuous, i.e., 
there exists a constant
$L>0$ such that
\begin{equation}
\|g(x)-g(y)\| \le L \|x-y\|,
\label{function}
\end{equation}
for all $x, y \in {\cal M}$.
\item[2.] There are positive numbers $\delta>0$, $1>m >0$, $M>1$, 
and a neighborhood of $\bar{x}$, defined by 
${\cal N}(\bar{x})=\{x: f({x})-f(\bar{x}) \le \delta \}$, such that for all
$x \in {\cal N}(\bar{x})$ and for all $z \in \R^n$,
\begin{equation}
m \|z\|^2 \le z^{\T} H(x) z \le M \|z\|^2.
\label{convex}
\end{equation}
\item[3.] There is a positive number $L >0$ such that for all $x \in {\cal N}(\bar{x})$,
\begin{equation}
\| H(x) - H(\bar{x}) \| \le L\|x-\bar{x}\|.
\label{Matrix}
\end{equation}
\end{itemize}
All these assumptions are standard for proving global and superlinear 
convergence for the BFGS algorithm (and similarly
for the robust BFGS algorithm). Assumption 1 
\cite[(3.13)]{nocedal99} is required when we use the Zoutendijk 
theorem to establish the global convergence for the robust BFGS
algorithm. Assumption 2 \cite[Assumption 8.1]{nocedal99} 
indicates that for all $x \in {\cal N}(\bar{x})$, a 
strong second order sufficient condition holds, i.e., there is a 
unique minimum in the neighborhood ${\cal N}(\bar{x})$, which 
will be used to prove the global\footnote{We say an algorithm
is globally convergent if the convergence does not depend
on the selection of the initial point.} and superlinear convergence. 
The real number of $m$ and $M$ are also used to choose the 
coefficient of the convex combination (see Eq. (16) and the discussion
about the selection of $\gamma_k$). Assumption 3 
\cite[Assumption 8.2]{nocedal99} will be 
needed only for the proof of the superlinear convergence. 
The parameter $L$ in (\ref{function}) may be different from 
the parameter $L$ in (\ref{Matrix}). But we can always choose 
the largest value of $L$ so that (\ref{function}) and (\ref{Matrix}) 
will hold for the same $L$, which will simplify our notation and proof.
It is worthwhile to point out that all functions which have isolated local
minimums and smooth enough will be in the class that satisfies 
these assumptions. 

For most optimization algorithms, the search for the minimizer of
(\ref{obj}) is carried out by using 
\begin{equation}
x_{k+1}=x_k+\alpha_k d_k,
\label{updateX}
\end{equation}
where $\alpha_k$ is the step size, and $d_k$ is the search direction. 
For Newton's method, the search direction $d_k$ is defined by
\begin{equation}
H(x_k) d_k = - g(x_k),
\label{dNewton}
\end{equation}
and the step size is set to $\alpha_k =1$. If Newton's method 
converges, it converges fast (quadratic) but (a) it may not converge 
at all, and (b) the computation of $H(x_k)$ is expensive.

To overcome the second problem associated with Newton's method,
the BFGS algorithm is developed to reduce the cost of the computation
of $H(x_k)$ while retaining the feature of fast (superlinear) 
convergence. It estimates $H(x_k)$ using the following update 
formula
\begin{equation}
B_{k+1}=B_k-\frac{B_ks_ks_k^{\T}B_k}{s_k^{\T}B_ks_k}+\frac{y_ky_k^{\T}}{y_k^{\T}s_k},
\label{BFGS}
\end{equation}
where 
\begin{equation}
s_k=x_{k+1}-x_k, \,\,\,\, y_k=g(x_{k+1})-g(x_k).
\label{skyk}
\end{equation}
The BFGS update has been very successful because its efficiency is
proved in theory and demonstrated in computational practice.
However, the first problem associated with Newton's method
still exists for the BFGS algorithm, i.e., the BFGS algorithm
may not converge for nonlinear non-convex problems \cite{dai02}.

To overcome the first problem associated with Newton's method, 
in \cite{yyang13}, a robust 
Hessian $(\gamma_k I +(1-\gamma_k) H(x_k))$ is suggested 
and the search for a minimizer is carried out along a direction 
$d_k$ that satisfies 
\begin{equation}
(\gamma_k I +(1-\gamma_k) H(x_k)) d_k = Q(x_k) d_k = - g(x_k),
\label{direction}
\end{equation}
where $\gamma_k \in [0,1]$ is carefully selected in every iteration. 
Clearly, the robust Hessian $Q(x_k)$ is a convex combination of 
the identity matrix for the steepest descent algorithm and the 
Hessian for Newton's algorithm. When $\gamma_k=1$, the 
algorithm reduces to the steepest descent algorithm; when 
$\gamma_k=0$, the algorithm reduces to Newton's algorithm. 
The global and quadratic convergence for the robust 
Newton algorithm is established in \cite{yyang13}.
However, the second problem associated with Newton's method
still exists for the robust Newton algorithm, i.e., the iteration of
the robust Newton algorithm is very expensive because of the
computation of $H(x_k)$.

To simultaneously fix both convergence and efficiency problems
associated with the Newton's algorithm, we propose a robust 
BFGS algorithm to estimate the robust Hessian $Q(x_k)$. 
We show that the computational cost of the robust BFGS 
in each iteration is roughly the same as computational cost 
of the BFGS, and the robust BFGS algorithm is globally 
and superlinearly convergent for all nonlinear unconstrained 
problems, convex or non-convex, under the conditions that 
Assumptions 1-3 hold.

Note that the BFGS updates $B_{k+1}$, an estimation of 
$H(x_{k+1})$, which is derived from the secant equation
\begin{equation}
y_k=B_{k+1} s_k.
\label{secant1}
\end{equation}
From (\ref{direction}) and (\ref{secant1}), the robust BFGS 
update $E_{k+1}$, an estimation of 
$Q(x_{k+1})=\gamma_k I +(1-\gamma_k) H_{k+1}$, is
derived from a modified secant equation
\begin{equation}
z_k=E_{k+1} s_k,
\label{secant0}
\end{equation}
where $z_k$ satisfies
\begin{equation}
z_k=E_{k+1} s_k=(\gamma_k I +(1-\gamma_k) B_{k+1})
s_k =\gamma_ks_k+(1-\gamma_k) y_k,
\label{secant2}
\end{equation}
and $\gamma_k \in [0,1]$ will be carefully selected in every iteration. If
$\gamma_k=1$, $E_{k+1}=I$ estimates $Q(x_{k+1})$ and the 
robust BFGS reduces to the steepest descent 
method from (\ref{direction}). If $\gamma_k=0$, $E_{k+1}=B_{k+1}$ 
estimates $Q(x_{k+1})=H_{k+1}$ and the 
robust BFGS reduces to the BFGS method.

It is straightforward to derive the robust BFGS formula from 
(\ref{secant2}) following exactly the same procedures of 
\cite[pages 197-198]{nocedal99}, which gives:
\begin{equation}
E_{k+1}=E_k-\frac{E_ks_ks_k^{\T}E_k}{s_k^{\T}E_ks_k}
+\frac{z_kz_k^{\T}}{z_k^{\T}s_k}.
\label{robustBFGS}
\end{equation}
Using the Sherman-Morrison-Woodbury formula 
\cite[page 51]{gv89}, we have
\begin{equation}
E_{k+1}^{-1}=
\left(I-\frac{s_{k}z_{k}^{\T}}{z_{k}^{\T}s_{k}}\right)E_{k}^{-1}
\left(I-\frac{z_{k}s_{k}^{\T}}{z_{k}^{\T}s_{k}}\right)
+\frac{s_{k}s_{k}^{\T}}{z_{k}^{\T}s_{k}}.
\label{E1}
\end{equation}
Therefore, similar to the selection of Newton's direction of 
(\ref{dNewton}), the robust BFGS search direction is defined by
\begin{equation}
E_k d_k=-g_k,
\label{dirMBFGS}
\end{equation}
hence, the search direction $d_k$ is calculated by 
\begin{equation}
d_k=-E_k^{-1} g_k,
\label{calMBFGS}
\end{equation}
where $E_k^{-1}$ is updated using ({\ref{E1}).

To apply the Zoutendijk theorem (see Theorem \ref{Zoutendijk}) 
in the global convergence 
analysis in the next section, $d_k$ is desired to be a descent 
direction, this requires $E_k \succ 0$ for all $k \ge 0$. 
Assume $E_0\succ 0$ is selected, and $E_{k-1}\succ 0$ is 
obtained. Since $d_{k-1}$ is a descent direction from 
(\ref{dirMBFGS}), it follows that $x_{k} \neq x_{k-1}$, hence,
$s_{k-1}=x_{k}-x_{k-1} \neq 0$. We will
select $\gamma_{k-1}$ such that $z_{k-1}^{\T}s_{k-1}>0$, 
Under this condition, in view of (\ref{E1}), for any 
$0 \neq v \in \R^n$, if $v^{\T}s_{k-1}=0$, then  
\[
v^{\T}E_{k}^{-1}v =v^{\T}E_{k-1}^{-1}v >0,
\]
if $v^{\T}s_{k-1} \neq 0$, then  
\[
v^{\T}E_{k}^{-1}v \ge \frac{(v^{\T} s_{k-1})^2}
{z_{k-1}^{\T}s_{k-1}} >0,
\]
i.e., $v^{\T}E_{k-1}^{-1}v >0$. Therefore,
$E_{k}\succ 0$, i.e., $d_k$ is indeed a descent direction for all 
$k \ge 0$.

As a matter of fact, we want to select $\gamma_k$ 
to meet stronger conditions:
\begin{equation}
m \le \frac{z_k^{\T}s_k}{s_k^{\T}s_k} \,\,\,\,\, \text{and} \,\,\,\,\,
\frac{z_k^{\T}z_k}{z_k^{\T}s_k} \le M ,
\label{conditions}
\end{equation}
where $m$ and $M$ are consistent to the ones in Assumption 2 
and they satisfy $0<m<1<M<\infty$. Inequalities in 
(\ref{conditions}) will be used to prove the global convergence 
of the robust BFGS algorithm. 

In view of (\ref{secant2}), the first
inequality of (\ref{conditions}) can be rewritten as 
\begin{eqnarray}
& z_k^{\T}s_k=(\gamma_k s_k^{\T} +(1-\gamma_k)y_k^{\T} )s_k 
\ge m s_k^{\T}s_k \nonumber \\
\iff & \gamma_k(s_k^{\T}s_k-y_k^{\T}s_k)\ge ms_k^{\T}s_k-y_k^{\T}s_k.
\label{add1}
\end{eqnarray}
Since $m<1$, we have 
\begin{equation}
s_k^{\T}s_k-y_k^{\T}s_k > ms_k^{\T}s_k-y_k^{\T}s_k.
\label{add2}
\end{equation}
First, if $s_k^{\T}s_k-y_k^{\T}s_k =0$, we have 
$0>ms_k^{\T}s_k-y_k^{\T}s_k$, and inequality (\ref{add1}) holds 
for any $\gamma_k \in [0,1]$. Second, if $s_k^{\T}s_k-y_k^{\T}s_k \neq0$, 
we denote
\begin{equation}
\check{\gamma}_k=\frac{ms_k^{\T}s_k-y_k^{\T}s_k}{s_k^{\T}s_k-y_k^{\T}s_k}
\label{gamma0}
\end{equation}
and divide our discussion into two cases: Case (a)
if $s_k^{\T}s_k-y_k^{\T}s_k >0$, since $s_k^{\T}s_k>m s_k^{\T}s_k$,
it follows that $\check{\gamma}_k<1$; since we select 
$\gamma_k \in [0,1]$, from (\ref{add1}), it follows that
$1 \ge {\gamma}_k \ge \max \{ \check{\gamma}_k, 0 \}$. 
Case (b) if $s_k^{\T}s_k-y_k^{\T}s_k <0$, we have 
$m s_k^{\T}s_k-y_k^{\T}s_k <s_k^{\T}s_k-y_k^{\T}s_k <0$ 
which shows $\check{\gamma}_k>1$.
Since $\gamma_k \in [0,1]$, summarizing the above discussion
shows that the first inequality of (\ref{conditions}) is equivalent to
\begin{equation}
\left\{ 
\begin{array}{ll}
\max \{ \check{\gamma}_k, 0 \} \le \gamma_k \le 1 &
\text{if }  s_k^{\T}s_k > y_k^{\T}s_k , \\
0 \le \gamma_k \le 1,  &
\text{if } s_k^{\T}s_k = y_k^{\T}s_k, \\
0 \le \gamma_k \le 1 \le \check{\gamma}_k, &
\text{if } s_k^{\T}s_k < y_k^{\T}s_k.
\end{array} \right.
\label{gamma1}
\end{equation}

In view of (\ref{secant2}), the second inequality of (\ref{conditions}) 
can be rewritten as
\begin{subequations}
\begin{align}
& z_k^{\T}z_k=(\gamma_k s_k +(1-\gamma_k)y_k )^{\T}(\gamma_k s_k +(1-\gamma_k)y_k ) 
\le M (\gamma_k s_k +(1-\gamma_k)y_k )^{\T}s_k \label{a}\\
\iff & p(\gamma_k) \equiv \gamma_k^2(s_k-y_k)^{\T}(s_k-y_k) 
+ \gamma_k(s_k-y_k)^{\T}(2y_k-Ms_k)+y_k^{\T}(y_k-Ms_k) \le 0. \label{b}
\end{align}
\label{Mcondition}
\end{subequations}
If $s_k = y_k$, then, inequality (\ref{b}) reduces to 
$(1-M)y_k^{\T}y_k \le 0$ which holds for 
$\forall {\gamma}_k \in [0,1]$. If $s_k \neq y_k$, 
$p(\gamma_k)$ is a quadratic and convex function of 
$\gamma_k$. Since $M>1$, it is
easy to see that the strict inequality of (\ref{a}) holds for $\gamma_k=1$, hence 
$p(1)<0$. Therefore $p(\gamma_k)=0$ has two solutions $\underline{\gamma}_k$ and 
$\bar{\gamma}_k$ satisfying $\underline{\gamma}_k<1<\bar{\gamma}_k$ and for any 
$\gamma_k \in [\underline{\gamma}_k, \bar{\gamma}_k]$, (\ref{a}) holds.
From (\ref{b}), if $s_k \neq y_k$,
\begin{equation}
\begin{array}{ll}
\underline{\gamma}_k & = \frac{(s_k-y_k)^{\T}(Ms_k-2y_k)-
\sqrt{((s_k-y_k)^{\T}(Ms_k-2y_k))^2-4(s_k-y_k)^{\T}(s_k-y_k)y_k^{\T}(y_k-Ms_k)}}
{2(s_k-y_k)^{\T}(s_k-y_k)}  \\
& = \frac{(s_k-y_k)^{\T}(Ms_k-2y_k)-
\sqrt{(M(s_k^{\T}(s_k-y_k))^2+4(M-1)(s_k^{\T}s_ky_k^{\T}y_k-(y_k^{\T}s_k)^2)}}
{2(s_k-y_k)^{\T}(s_k-y_k)}.
\end{array}
\label{gamma2}
\end{equation}

Since $\gamma_k \in [0,1]$, we have
\begin{equation}
\max \{ 0, \underline{\gamma}_k \} \le \gamma_k \le 1.
\label{gamma3}
\end{equation} 
Intuitively, it is desirable to have $\gamma_k \in [0,1]$ as close 
to zero as possible so that the algorithm will approach to the 
standard BFGS algorithm. Therefore, we want
to select the smallest $\gamma_k \in [0,1]$ satisfying 
(\ref{gamma1}) and (\ref{gamma3}).
We consider all possible relations among $ms_k^{\T}s_k$,  $s_k^{\T}s_k$, and $y_k^{\T}s_k$. Since $m<1$, we have 
$ms_k^{\T}s_k < s_k^{\T}s_k$.
\begin{itemize}
\item{Case 1 ($y_k^{\T}s_k < ms_k^{\T}s_k < s_k^{\T}s_k$)}: 
To select the smallest $\gamma_k \in [0,1]$ satisfying 
(\ref{gamma1}) and (\ref{gamma3}), combining the first relation 
of (\ref{gamma1}) and (\ref{gamma3}), we have 
\[
{\gamma}_k = \max \{\max \{0,  \check{\gamma}_k \} ,\underline{\gamma}_k \}.
\]
From (\ref{gamma0}), we know  $\check{\gamma}_k>0$ in this 
case. Therefore, $ {\gamma}_k = \max \{ \check{\gamma}_k ,\underline{\gamma}_k \}$.
\item{Case 2 ($ ms_k^{\T}s_k \le y_k^{\T}s_k < s_k^{\T}s_k$)}: 
To select the smallest $\gamma_k \in [0,1]$ satisfying 
(\ref{gamma1}) and (\ref{gamma3}), combining the first 
relation of (\ref{gamma1}) and (\ref{gamma3}), we have 
\[
{\gamma}_k = \max \{\max \{0,  \check{\gamma}_k \} ,\underline{\gamma}_k \}.
\]
From (\ref{gamma0}), we know  $\check{\gamma}_k \le 0$ in this 
case. Therefore, 
$ {\gamma}_k = \max \{ 0 ,\underline{\gamma}_k \}$.
\item{Case 3 ($ ms_k^{\T}s_k < s_k^{\T}s_k \le y_k^{\T}s_k$)}: 
To select the smallest $\gamma_k \in [0,1]$ 
satisfying (\ref{gamma1}) and (\ref{gamma3}), combining the 
last 2 relations of (\ref{gamma1}) and (\ref{gamma3}), we have 
$ {\gamma}_k = \max \{ 0 ,\underline{\gamma}_k \}$. 
In particular, when $s_k = y_k$, we select ${\gamma}_k = 0$.
\end{itemize}
Combining all cases, we have
\begin{equation}
{\gamma}_k = \left\{ \begin{array}{lr}
\max \{\underline{\gamma}_k, \check{\gamma}_k \}, & \text{if } ms_k^{\T}s_k > y_k^{\T}s_k, \\
\max \{0, \underline{\gamma}_k \}, &  \text{if } ms_k^{\T}s_k \le y_k^{\T}s_k, \\
0, & \text{if } s_k = y_k.
\end{array} \right.
\label{gamma}
\end{equation}

\begin{remark}
Equation (\ref{E1}) can be replaced by the following equivalent 
representation 
\begin{equation}
E_{k+1}^{-1}=E_{k}^{-1}
-\frac{s_{k}z_{k}^{\T}}{z_{k}^{\T}s_{k}}E_{k}^{-1}
-E_{k}^{-1}\frac{z_{k}s_{k}^{\T}}{z_{k}^{\T}s_{k}}
+\frac{s_{k}}{z_{k}^{\T}s_{k}}\left( z_{k}^{\T}E_{k}^{-1}z_{k} \right) \frac{s_{k}^{\T}}{z_{k}^{\T}s_{k}}
+\frac{s_{k}s_{k}^{\T}}{z_{k}^{\T}s_{k}}.
\label{badImp}
\end{equation}
which requires fewer computational counts than (\ref{E1}) does. 
However, this equivalent formula is not numerically stable. 
For CUTEst problem {\tt heart6ls}, when the condition number of 
$E_{k+1}^{-1}$ is poor, two of the eigenvalues of 
$E_{k+1}^{-1}$ generated by (\ref{badImp}) are negative, 
but all eigenvalues of $E_{k+1}^{-1}$ generated by (\ref{E1}) 
are greater than zero. This means that formula (\ref{E1}) 
is more robust than formula (\ref{badImp}).
\end{remark}
\begin{remark}
Although the two formulas in (\ref{gamma2}) are equivalent, 
the second one is numerically much more stable. For the 
CUTEst test problem {\tt djkl}, using the first formula results 
in a negative value inside the square root because of the 
computational error, while the second formula ensures a 
positive value inside the square root.
\end{remark}
\begin{remark}
It is worthwhile to note that if $y_k^{\T}y_k \le My_k^{\T}s_k$ 
holds, then (\ref{a}) holds for ${\gamma}_k=0$, i.e., 
$\underline{\gamma}_k \le 0$. In addition, if 
$ms_k^{\T}s_k \le y_k^{\T}s_k$ holds at the same time, 
then from (\ref{gamma}), $\gamma_k=0$. Moreover, $d_k$ 
is a descent direction.
\label{first}
\end{remark}

Now we are ready to present the robust BFGS algorithm.
\begin{algorithm} {\bf robust BFGS} \\
\begin{algorithmic}[1] 
\STATE Data:  $0<\epsilon$, $m<1$, and $1<M<\infty$, initial $x_0$, and 		     
     $E_0=I$.
\FOR{ k=0,1,2,...}
	\STATE Calculate gradient $g(x_k)$. If $\|g(x_k)\|<\epsilon$, stop;
	\STATE Compute search direction $d_k$ using (\ref{calMBFGS});
	\STATE Set $x_{k+1}=x_k+\alpha_k d_k$, where $\alpha_k$ satisfies the
		Wolfe condition to be defined later;
	\STATE Compute $s_k$ and $y_k$ using (\ref{skyk});
	\STATE Select $\gamma_k$ using (\ref{gamma}), and compute $z_k$ 
		using (\ref{secant2});
	\STATE Update $E_{k+1}^{-1}$ using (\ref{E1});
	\STATE $k \leftarrow k+1$;
\ENDFOR
\end{algorithmic}
\label{newAlgo}
\end{algorithm}

\begin{remark}
It is clear that the computation involving the selection of 
$\gamma_k$ is negligible (requires ${\cal O}(n)$ operations). 
Therefore, the cost of robust BFGS in each iteration is almost
the same as the cost of the BFGS.
\label{efficiency}
\end{remark}
 
In the next section, our discussion will focus on the proof 
of global and superlinear convergence of Algorithm \ref{newAlgo}. 
The convergence properties are directly related to the goodness 
of the search direction and step length. The quality of the search 
direction is measured by
\begin{equation}
\cos(\theta_k)=-\frac{g_k^Td_k}{\|g_k\| \|d_k\|}=\frac{d_k^{\T}E_ks_k}{\|E_kd_k\|\|s_k\|},
\label{cos}
\end{equation}
where the second equation follows from (\ref{updateX}) 
and (\ref{dirMBFGS}). A good step length $\alpha_k$ should 
satisfy the following Wolfe condition.
\begin{subequations}
\begin{align}
f(x_k+\alpha_k d_k) \le f(x_k)+\sigma_1 \alpha_kg_k^Td_k, \\
d_k^{\T}g(x_k+\alpha_k d_k) \ge \sigma_2 g_k^Td_k,
\end{align}
\end{subequations}
where $0 < \sigma_1 < \sigma_2<1$. The existence of the 
Wolfe condition is established in \cite{wolfe69,wolfe71}.
An algorithm that finds, in finite steps, a point satisfying 
the Wolfe condition is given in \cite{more90}. Therefore, 
we will not discuss step size selection in this paper.

\section{Convergence Analysis}

An important global convergence result was given by 
Zoutendijk \cite{zoutendijk70} which can be stated as follows.
\begin{theorem}
Suppose that $f$ is bounded below in $\R^n$ and that $f$ is 
continuously twice differentiable in an open set ${\cal M}$ 
containing the level set ${\cal L}=\{ x: f(x) \le f(x_0) \}$. 
Assume that the gradient is Lipschitz continuous on 
${\cal M}$, i.e., there exists a constant $L>0$ such that
\begin{equation}
\|g(x)-g(y)\| \le L \|x-y\|,
\end{equation}
for all $x, y \in {\cal M}$. Assume further that $d_k$ is a 
descent direction and $\alpha_k$ satisfies the Wolfe condition. 
Then
\begin{equation}
\sum_{k \ge 0} \cos^2(\theta_k) \| g_k \|^2 < \infty.
\end{equation}
\hfill \qed
\label{Zoutendijk}
\end{theorem}
The Zoutendijk theorem indicates that if $d_k$ is a descent 
direction and $\cos(\theta_k) \ge \delta >0$, for all $k$, then 
the algorithm is globally convergent because 
$\lim_{k \rightarrow \infty} \|g_k\|=0$. 

Let $\{ \eta_k \}$ and $\{ \nu_k \}$ be two scalar nonnegative 
infinite sequences, we define $\eta_k=o(\nu_k)$ \cite[Page 591]{nocedal99}
if the ratio 
$\{ \eta_k / \nu_k \}$ approaches zero, that is 
\begin{equation}
\lim_{k \rightarrow \infty} \frac{\eta_k}{\nu_k} =0.
\label{littleO}
\end{equation}
Now we are ready to state the main convergence result for 
the robust BFGS algorithm.

\begin{theorem}
Suppose that $f$ is bounded below in $\R^n$, $f$ is continuously 
twice differentiable in an open set ${\cal M}$ containing the 
level set ${\cal L}=\{ x: f(x) \le f(x_0) \}$,
and Assumptions 1-3 hold. Then the sequence generated by 
Algorithm \ref{newAlgo} is globally convergent in the 
sense that $\lim \inf \| g(x_k) \| \rightarrow 0$. In addition, 
if $\| s_k \|=o(1/k^{1+\epsilon})$ for any $\epsilon >0$, 
then the sequence generated by 
Algorithm \ref{newAlgo} converges to some local minimum
point $\bar{x}$ satisfying $\| g(\bar{x}) \|=0$ with superlinear rate.
\end{theorem}
\begin{proof}
First, we show that $d_k$ is a descent direction. From the 
selection of $\gamma_k$ using (\ref{gamma}), 
we know that (\ref{conditions}) holds. The first inequality 
of (\ref{conditions}) guarantees $z_k^{\T}s_k >0$. 
Using this fact and (\ref{E1}), we conclude $E_k\succ 0$ 
since $E_0=I \succ 0$. Therefore, $d_k$ is a descent 
direction. Since (\ref{conditions}) holds for all $k \ge 0$, 
following exactly the same arguments used in the proof of 
\cite[Theorem 8.5]{nocedal99}, we have  a sub-sequence 
$\{ j_k \}$ such that
\begin{equation}
\cos(\theta_{j_k}) \ge \delta > 0.
\label{subseq}
\end{equation}
According to Step 5 of Algorithm \ref{newAlgo}, $\alpha_k$ is
selected to satisfy the Wolfe conditions. In view of 
Theorem~\ref{Zoutendijk}, we have 
\begin{equation}
\sum_{k \in  \{ j_k \}} \cos^2(\theta_k) \| g_k \|^2 < \infty.
\label{liminf}
\end{equation}
Using the notation of $\lim \inf$ defined in \cite[Page 578]{nocedal99}, 
inequalities (\ref{subseq}) and (\ref{liminf}) indicate
\begin{equation}
\lim \inf \| g(x_k) \| \rightarrow 0.
\label{limit}
\end{equation}
Denote the $i$-th component of $s_k$ by  $s_i^k$ and
the $i$-th component of $x_k$ by  $x_i^k$. Since
$\| s_k \|=o(1/k^{1+\epsilon})$, we have $| s_i^k |=o(1/k^{1+\epsilon})$. 
For the $i$-th component $x_i^{k+1}$ of $x_{k+1}$, we have 
$x_i^{k+1}=x_i^0+\sum_{j=0}^{k} s_i^j$. To show that
$\{ x_i^{k+1} \}$ is convergent, we just need to show
$\lim_{k \rightarrow 0} \sum_{j=0}^{k} s_i^j$ is convergent.
The latter is guaranteed if 
\begin{equation}
\lim_{k \rightarrow 0} \sum_{j=0}^{k} | s_i^j|
\label{absoluteC}
\end{equation}
is convergent. Since $| s_i^k |=o(1/k^{1+\epsilon})$ holds,
i.e., every component 
$\{ x_i^{k} \}$ of $\{ x_{k} \}$ is convergent. Therefore, this
shows that $x_{k} \rightarrow \bar{x}$ and $\| g(\bar{x}) \| = 0$. 
In view of the Assumption 2, the function is locally 
strongly convex in a neighborhood of 
$\bar{x}$ which satisfies $\| g(\bar{x}) \| = 0$, this means that 
$\bar{x}$ is an isolated local minimizer in the neighborhood. 
This shows that the iterates will converge to a local minimizer.

Since Algorithm \ref{newAlgo} is globally convergent and
the convergent point is a local minimizer, for sufficiently large $k$
and for $\delta$ defined in Assumption 2, we have
$f(x_k) \le f(\bar{x}) +\delta$. Therefore, for all $v \in \R^n$,
\begin{equation}
m \|v\|^2 \le v^{\T} H(x_k) v \le M \|v\|^2.
\label{convexk}
\end{equation}
Using Taylor's Theorem \cite[Theorem 2.1]{nocedal99} 
\begin{equation}
y_k=g(x_{k+1})-g(x_k)=\int_{0}^{1} H(x_k+t\alpha_kd_k) s_k\mathrm{d}t \equiv \bar{H}_ks_k,
\label{taylor}
\end{equation}
then (\ref{convexk}) and (\ref{taylor}) imply that 
$\bar{H}_k$ is positive definite, i.e., for all $v \in \R^n$,
\begin{equation}
m \|v\|^2 \le v^{\T} \bar{H}_k v \le M \|v\|^2.
\label{convexbar}
\end{equation}
This gives 
\[
\frac{y_k^{\T}s_k}{s_k^{\T}s_k}=\frac{s_k^{\T}\bar{H}_k s_k}{s_k^{\T}s_k}=
\frac{s_k^{\T}}{\|s_k\|}\bar{H}_k \frac{s_k}{\|s_k\|} \ge m,
\]
and
\[
\frac{y_k^{\T}y_k}{y_k^{\T}s_k}=\frac{s_k^{\T}\bar{H}_k^2s_k}{s_k^{\T}\bar{H}_k s_k}
=\frac{(\bar{H}_k^{\frac{1}{2}}s_k)^{\T}}{\|\bar{H}_k ^{\frac{1}{2}}s_k\|}\bar{H}_k 
\frac{\bar{H}_k^{\frac{1}{2}}s_k}{\|\bar{H}_k^{\frac{1}{2}}s_k\|}
 \le M.
\]
From these two inequalities, in view of Remark \ref{first}, we 
conclude that for large $k$, $\gamma_k=0$ is always selected. 
Therefore, the robust BFGS reduces to the standard BFGS 
for large $k$. In addition, if Assumption 3 holds,
the BFGS converges at a superlinear rate 
\cite[Theorem 8.6]{nocedal99}, therefore the robust BFGS 
also converges at the superlinear rate because it is identical 
to the BFGS for large $k$. 
\hfill \qed
\end{proof}

\begin{remark}
The additional condition $\| s_k \|=o(1/k^{1+\epsilon})$ is sufficient. A 
less restrictive condition $s_k \rightarrow 0$ is necessary.
Considering the function $f(x)=e^x-e^{-x^2}$, which has a
minimum at $x \approx -0.3938$, if the initial point is at
$x_0 = -3$, then $x_k \rightarrow -\infty$ and $g(-\infty)=0$.
But $x_k  \rightarrow -\infty$ is not a local minimizer because
$s_k \rightarrow 0$ does not hold.
\end{remark}

\section{Implementation and Numerical Test}

This section provides detailed information about the implementation 
of the algorithm which is slightly different from the description of 
Algorithm \ref{newAlgo}. It also presents our test results for 
CUTEst problems.

\vspace{0.15in}
\subsection{Implementation details}

Algorithm \ref{newAlgo} has been implemented in Matlab 
function {\tt mBFGS} with the following considerations.
First, the selection of $m$ and $M$ turns out to be important. 
The $m$ and $M$ of $H(\bar{x})$ satisfying (\ref{convex}) 
depend on the individual function to be optimized and each 
of its local minimizers. To be safe, one may select small $m$ 
and large $M$, which will be likely cover all possible functions 
to be optimized, but this selection may not be numerically robust. 
On the other hand, if $m$ is selected too big, and/or $M$
is selected too small, the selection may violate condition 
(\ref{convex}). Our selection of the default set of parameters 
are $m=0.00001$, $M=100000$, and $\epsilon=0.00001$,
which, in general, give very impressive computational result. 

Second, for several test problems, the condition numbers of the 
estimated $E_k^{-1}$ are poor at the early iterations, which 
leads to very large vector $d_k$. If this happens, line search
takes long time to find a better iterate. Therefore, $d_k$ is 
re-scaled to $\| d_k \| = 10^6$ if $\| d_k \| > 10^6$ 
is detected. 

Test result on the algorithm with the above implementation is 
in general very impressive. However, for a few problems, 
it takes many iterations to converge, which may be due to the 
poor estimation of $m$ and $M$ by the default set of the 
parameters (please note that $m$ and $M$ are used to 
select $\gamma_k$). Therefore, a modification that dynamically 
selects $m_k$ and $M_k$ is implemented in {\tt mBFGS} 
for the purpose of getting better estimation of the bounds 
of a particular local minimizer of a particular function. 

From (\ref{gamma0}) and (\ref{gamma2}), it is easy to see that 
$m$ only affects the value of $\check{\gamma}_k$, and $M$ 
only affects the value of $\underline{\gamma}_k$. We want 
to select $m$ and $M$ such that $\check{\gamma}_k$ and 
$\underline{\gamma}_k$ become as small as possible to 
maximize the chance of using the BFGS formula. This requires 
the selection of small $m$ and large $M$, or the selection 
of the ratio of $M$ to $m$ as large as possible. 
On the other hand, we noticed, from (\ref{secant2}), 
that $m$ and $M$ together affect the 
condition number of $E_{k+1}$, which is important to the 
numerical robustness in the computation of $d_{k+1}$ from 
(\ref{dirMBFGS}). This requires the ratio of $M$ to $m$ 
as small as possible.  For the trade off, we select the ratio 
of $M$ to $m$  as $10^{10}$. The nominal parameters in 
{\tt mBFGS} are $\bar{m}=0.00001$, $\bar{M}=100000$, and 
$\epsilon=0.00001$. Because of the fixed ratio of $M$ to $m$, 
we must increase or decrease $m$ and $M$ at the same time. 
From (\ref{gamma0}) and (\ref{gamma2}), increasing $m$ 
and $M$ will decrease $\underline{\gamma}_k$
but increase $\check{\gamma}_k$; and decreasing $m$ and 
$M$ will increase $\underline{\gamma}_k$ but decrease 
$\check{\gamma}_k$. From (\ref{gamma}), we want the 
difference between $\check{\gamma}_k$ and  
$\underline{\gamma}_k$ to be small, so that the final
choice $\gamma_k$ is small. Therefore, we dynamically 
adjust $m$ and $M$ using the following simple heuristics.

\begin{algorithm} {Selection of $m$ and $M$ } \\
\begin{algorithmic}[1] 
\STATE Set $m=\bar{m}$, $M=\bar{M}$, and calculate $\check{\gamma}_k$.
\IF{$\check{\gamma}_k>1$}
	\STATE   adjust $M=10^4\bar{M}$.
	\STATE   calculate $\underline{\gamma}_k$.
\ELSIF{$\check{\gamma}_k \le 1$}
	\STATE   calculate $\underline{\gamma}_k$.
	\IF{ $\underline{\gamma}_k-\check{\gamma}_k > 0.2$
	and $\underline{\gamma}_k >0$}
		\STATE  adjust $M=10^3\bar{M}$ and $m=10^3\bar{m}$.
		\STATE  recalculate $\underline{\gamma}_k$ and $\check{\gamma}_k$.
	\ELSIF{$\check{\gamma}_k-\underline{\gamma}_k> 0.2$ and $\check{\gamma}_k >0$}
		\STATE  adjust $M=10^{-2}\bar{M}$ and $m=10^{-2}\bar{m}$.
		\STATE  recalculate $\underline{\gamma}_k$ and $\check{\gamma}_k$.
	\ENDIF
\ENDIF
\end{algorithmic}
\label{mMAlg}
\end{algorithm}

This algorithm will be applied between Lines 6 and 7 of 
Algorithm \ref{newAlgo}, and is executed before the
calculation of (\ref{gamma}).
The above modification significantly reduces the number of 
iterations for the problems which had slow convergence
when fixed $m$ and $M$ were used. Moreover, it has little impact 
on the remaining problems.

\subsection{Numerical test}

The robust BFGS implementation {\tt mBFGS} and
the BFGS algorithm implemented in Matlab Optimization 
Toolbox function {\tt fminunc} are tested using the 
{CUTEst} test problem set. {\tt fminunc} options are set as
\footnotesize
\[
\mbox{options = optimset('LargeScale','off','MaxFunEvals',1e+20,'MaxIter',
5e+5,'TolFun',1e-20, 'TolX',1e-10).}
\]
\normalsize
This setting is selected to ensure that the BFGS implementation {\tt fminunc} 
will have enough iterations either to converge or to fail.  

We conducted tests for both {\tt mBFGS} and {\tt fminunc} 
using the {CUTEst} test problem set, which is downloaded 
from the Princeton test problem collections \cite{web}. 
Since the {CUTEst} test set is presented in AMPL mod-files, 
we first convert AMPL mod-files into nl-files so that Matlab 
functions can read the CUTEst models, then we use Matlab 
functions {\tt mBFGS} and {\tt fminunc} to read the 
nl-files and solve these test problems.
The objective function is calculated from AMPL command 
$\mbox{[f,c] = amplfunc(x,0)}$.
The gradient function is calculated from AMPL command 
$\mbox{[g,Jac] = amplfunc(x,1)}$.
Both {\tt mBFGS} and {\tt fminunc} use these values in 
the optimization algorithms. 
The test uses the initial points provided by the {CUTEst} 
test problem set, we record the calculated objective 
function values, the norms of the gradients at the final 
points, and the iteration numbers for these tested problems. 
We present this test results in Table 1. In this table, 
{\tt iter} stands for the number of total iterations used 
by the algorithm; {\tt obj} is the value of the objective 
function achieved at the end of the iterations; {\tt gradient} 
is the norm of the gradient achieved at the end of the iterations. 

Because the implementations of different algorithms do not
use the same language (Matlab is a interpreted language which
is slower than compiled language such as C/C++ and Fortran),
it is normally not easy to have a fair comparison for different
algorithms with different implementations. Fortunately, the 
analysis in Remark \ref{efficiency} indicated that the computational
cost of the robust BFGS in each iteration is almost the same
as the traditional BFGS, therefore, it makes sense to compare
the iteration counts among different algorithms.

Table 1 compares the traditional BFGS algorithm implemented
in Matlab optimization toolbox {\tt fminunc} and the robust BFGS 
implemented in {\tt mBFGS}.  (Please note that {\tt fminunc} is
also a compiled code but {\tt mBFGS} is an interpreted code. Also
the stopping criterion for {\tt mBFGS} is $\epsilon=0.00001$ .)

\footnotesize
\begin{center}
\begin{longtable}{|c|r|r|r|r|r|r|r|}
\caption{Test result for problems in CUTEst \cite{web}, initial points are given in CUTEst} \\
\hline    
Problem & size & iter   &  obj    & gradient  & iter    & obj     & gradient  \\
        & n    & mBFGS  &  mBFGS  & mBFGS     & fminunc & fminunc & fminunc    \\
\hline
arglina    & 100  &  1    & 100.0       & 0.0  & 4  & 100.0      &  0.1662e-03       \\ \hline
bard       & 3  &  18   & 0.8214e-02       & 0.2082e-06  & 20 & 0.8214e-02      &  0.1158e-05     \\ \hline
beale      & 2  &  13   & 0.3075e-14    & 0.1817e-06  & 15 & 0.2400e-09   &  0.1392e-06     \\ \hline
biggs6     & 6  &  70   & 0.1080e-09    & 0.7865e-05  & 68 & 0.4796e-03      &  {\bf 0.1802e-01} \\ \hline
box3       & 3  &  19   & 0.4077e-10    & 0.1828e-05  & 24 & 0.3880e-10   &  0.2364e-5     \\ \hline
brkmcc     & 2  &  4    & 0.1690e-00    & 0.8114e-07  & 5  & 0.1690e-00      &  0.4542e-06     \\ \hline
brownal    & 10  &  12   & 0.1406e-12   & 0.6968e-05  & 16 & 0.3050e-08   &  0.1043e-03      \\ \hline
brownbs    & 2  &  632  & 0.1950e-17    & 0.5369e-06  & 11 & 0.9308e-04      &  {\bf 15798.5950e-00} \\ \hline
brownden   & 4  &  21   & 85822.2016e-00    & 0.3135e-03  & 32 & 85822.2017e-00 &  {\bf 0.4646e-00} \\ \hline
chnrosnb   & 50  &  158  & 0.1686e-15    & 0.9838e-05  & 98 & 30.0583e-00      &  {\bf 10.1863e-00} \\ \hline
cliff      & 2  &  27   & 0.1997e-00   & 0.8919e-05  & 1  & 1.0015e-00      &  {\bf 1.4147e-00} \\ \hline
cube       & 2  &  21   & 0.4231e-19    & 0.6905e-09  & 34 & 0.7987e-09   &  0.1340e-03       \\ \hline
deconvu    & 51  &  80   & 0.3158e-06    & 0.1750e-03  & 80 & 0.3158e-06   &  0.1750e-03     \\ \hline
denschna   & 2  &  7    & 0.1467e-13    & 0.3346e-06  & 10 & 0.5000e-12   &  0.1581e-05     \\ \hline
denschnb   & 2  &  7    & 0.6047e-13    & 0.5505e-06  & 7  & 0.1000e-11   &  0.2200e-05     \\ \hline
denschnc   & 2  &  8    & {\bf 0.1833e-00} & 0.4104e-07  & 21 & 0.1608e-08   &  0.3262e-03     \\ \hline
denschnd   & 3  &  38   & 0.2461e-08    & 0.4040e-06  & 23 & 45.2971e-00      &  {\bf 84.5851e-00} \\ \hline
denschnf   & 2  &  10   & 0.4325e-17    & 0.3919e-07  & 10 & 0.2000e-10   &  0.1005028e-03     \\ \hline
dixon3dq   & 10  &  15   & 0.5626e-15    & 0.5588e-07  & 20 & 0.1400e-11   &  0.3661e-5     \\ \hline
djtl       & 2  &  79  & -8951.5447e-00       & 0.2881e-01  & 3  & -8033.8869e-00      &  {\bf 1273.3319e-00} \\ \hline
eigenals   & 110  &  79   & 0.7521e-13    & 0.3483e-05  & 78 & 0.1092e-02   &  {\bf 0.1029e-00} \\ \hline
eigenbls   & 110  &  513  & 0.1994e-10    & 0.7166e-05  & 91 & 0.3462e-00      &  {\bf 0.4642e-00} \\ \hline
engval2    & 3  &  27   & 0.1172e-12    & 0.9109e-95  & 29 & 0.3953e-09   &  0.2799e-03     \\ \hline
errinros   & 81  &  48   & 0.2442e-96    & 0.5141e-95  & 92 & 0.4577e-03      &  {\bf 0.2553e-00} \\ \hline
expfit     & 2  &  11   & 0.2405e-00       & 0.2350e-05  & 12 & 0.2405e-05      &  0.2263e-05     \\ \hline
extrosnb   & 10  &  1    & 0.0       & 0.0     & 1  & 0.0      &  0.0       \\ \hline
fletcbv2   & 100  &  97   & -0.5140e-00       & 0.9676e-05  & 98 & -0.5140e-00      &  0.1087e-4     \\ \hline
fletchcr   & 100  &  179  & 0.1114e-13     & 0.4440e-05  & 63 & 68.1289e-00       &  {\bf 160.9879e-00} \\ \hline
genhumps   & 5  &  47   & 0.1871e-09    & 0.8571e-05  & 59 & 0.4493e-07   &  0.3167e-03     \\ \hline
growthls   & 3  &  1    & {\bf 3542.1490e-00} & 0      & 12 & 12.4523e-00      &  {\bf 0.5809e-01} \\ \hline
hairy      & 2  &  18   & 20.0       & 0.5710e-05  & 22 & 20.0    &  0.3810e-04     \\ \hline
hatfldd    & 3  &  23   & 0.6615e-07    & 0.1853e-05  & 19 & 0.066150e-07   &  0.2355e-05     \\ \hline
hatflde    & 3  &  28   & 0.4434e-06    & 0.5321e-05  & 9  & 0.6210e-06   &  0.7970e-05     \\ \hline
heart6ls   & 6  & 2180  & 0.2620e-14    & 0.6696e-5  & 53 & 0.6318e-00      &  {\bf 71.9382548e-00} \\ \hline
helix      & 3  &  22   & 0.5489e-16    & 0.1901e-06  & 29 & 0.2260e-10   &  0.4196e-04     \\ \hline
hilberta   & 10  &  20   & 0.2422e-06    & 0.5563e-05  & 35 & 0.2289e-06   &  0.3263e-05     \\ \hline
hilbertb   & 50  &  11   & 0.4606e-12    & 0.3123e-05  & 6  & 0.21e-11   &  0.6542e-5     \\ \hline
himmelbb   & 2  &  1    & 0.9665e-13    & 0.1153e-06  & 6  & 0.1462e-04      &  0.1251e-02        \\ \hline
himmelbf   & 2  &  7    & 0.1069e-12    & 0.9337e-06  & 8  & 0.1000e-12   &  0.1448e-05     \\ \hline
himmelbg   & 2  &  7    & 0.1069e-12    & 0.9337e-06  & 8  & 0.1000e-12   &  0.1448e-05     \\ \hline
himmelbh   & 2  &  7    & -0.9999e-00   & 0.5188e-06  & 7  & -0.9999e-00  &  0.2607e-06     \\ \hline
humps      & 2  &  85   & 0.2281e-09    & 0.6755e-05  & 25 & 5.4248e-00      &  {\bf 2.3625e-00} \\ \hline
jensmp     & 2  &  1    & {\bf 2020}       & 0              & 16 & 124.3621e-00      &  0.2897e-05     \\ \hline
kowosb     & 4  &  27   & 0.3075e-03    & 0.1426e-05  & 33 & 0.3075e-03   &  0.1253e-06     \\ \hline
loghairy   & 2  &  87   & 0.1823e-00       & 0.6892e-06  & 11 &{\bf 2.5199e-00}&  0.5377e-02        \\ \hline
mancino    & 100  &  35   & 0.162e-16    & 0.9715e-05  & 9  & 0.2204e-02      &  {\bf 1.2243e-00} \\ \hline
maratosb   & 2  &  3    & -1.0       &  0.5142e-07  & 2  & -0.9997e-00      &  {\bf 0.3570e-01} \\ \hline
mexhat     & 2  &  1    & -0.4009e-01       & 0.6621e-05  & 4  & -0.4009e-01      &  0.13703e-04     \\ \hline

osborneb   & 11  &  51   & 0.4013e-01       & 0.2474e-05  & 76 & 0.4013e-01      &  0.7884e-05     \\ \hline
palmer1c   & 8  &  32   & 0.9759e-00     & 0.3995e-06  & 38 & 16139.4418e-00      &  {\bf 655.0159e-00} \\ \hline
palmer2c   & 8  &  86   & 0.1442e-01       & 0.1013e-05  & 60 & 98.0867e-00      &  {\bf 33.4524e-00} \\ \hline
palmer3c   & 8  &  47   & 0.1953e-01       & 0.8975e-05  & 56 & 54.3139e-00      &  {\bf 7.8518e-00} \\ \hline
palmer4c   & 8  &  77   & 0.5031e-01       & 0.6125e-05  & 56 & 62.2623e-00      &  {\bf 6.6799e-00} \\ \hline
palmer5c   & 6  &  12   & 2.1280e-00    & 0.2881e-05  & 14 & 2.1280e-00      &  0.7484e-03       \\ \hline
palmer6c   & 8  &  55   & 0.1638e-01       & 0.4847e-05  & 43 & 18.0992e-00      &  {\bf 0.7851e-00} \\ \hline
palmer7c   & 8  &  40   & 0.6019e-00   & 0.1289e-05  & 28 & 56.9098e-00      &  {\bf 4.0268e-00} \\ \hline
palmer8c   & 8  &  46   & 0.1597e-00   & 0.4974e-05  & 49 & 22.4365e-00      &  {\bf 1.3147e-00} \\ \hline
powellsq   & 2  & 0 & 0 & 0 & 0 & 0 & 0 \\ \hline
rosenbr    & 2  &  32   & 0.1382e-15    & 0.5153e-06  & 36 & 0.283e-10   &  2.6095e-5     \\ \hline
sineval    & 2  &  68   & 01271e-15    & 0.9558e-06  & 47 & 0.2212e-00      &  {\bf 1.2315e-00} \\ \hline
sisser     & 2  &  14   & 0.2809e-07    & 0.8680e-05  & 11 & 0.1540e-7   &  0.7282671e-5     \\ \hline
tointqor   & 50  &  37   & 1175.4722e-00   & 0.7734e-05  & 40 & 1175.4722e-00 &  0.9041e-07     \\ \hline
vardim     & 100  &  21   & 0.0001e-20    & 0.1126e-07  & 1  & 0.2244e-06   &  {\bf 0.5511e-00} \\ \hline
watson     & 31  &  48   & 0.2442e-06    & 0.5141e-05  & 90 & 0.1050e-02      &  {\bf 0.4875e-00} \\ \hline
yfitu      & 3  &  74   & 0.6670e-12    & 0.5396e-05  & 57 & 0.4398e-02      &  {\bf 11.8427e-00} \\ \hline
\end{longtable}
\label{BFGSvsmBFGS}
\end{center}
\normalsize

We summarize the comparison of the test result as follows:
\begin{itemize}
\item[1.] The robust BFGS function {\tt mBFGS} converges 
in all the test problems after terminate condition 
$\| g(x_k) \|< 10^{-5}$ is met except for three problems 
{\tt brownden}, {\tt deconvu}, and {\tt djtl}. But for these 
problems, {\tt mBFGS} finds better solutions than {\tt fminunc}. 
Moreover, for about $40\%$ of the problems, Matlab Toolbox  
BFGS function {\tt fminunc} does not reduce $\| g(x_k) \|$ 
to smaller than $0.01$. For these problems, the objective 
functions obtained by {\tt fminunc} normally 
are not close to the minimum;
\item[2.] For those problems that both {\tt mBFGS} and 
{\tt fminunc} converge, {\tt mBFGS} most time uses less 
iterations than {\tt fminunc} and converges to a point with 
smaller $\| g(x_k) \|$;
\item[3.] There are three problems ({\tt denschnc}, 
{\tt growthls}, and {\tt jensmp}), {\tt mBFGS} converges 
to a local minimum but {\tt fminunc} finds a better point.
\end{itemize}

An alternative way to see the superiority of the proposed 
algorithm is to use {\it performance profile} which, to 
our best knowledge, was first proposed 
in \cite{ty96} and carefully analyzed in \cite{dm02}. Let 
${\cal S}$ be the set of solvers and ${\cal P}$ be the set
of test problems. Let $n_s$ be the number of solvers, $n_p$ be 
the number of test problems, $m_{p,s}$ be the merit function 
(which can be for example the iteration numbers) 
of using solver $s$ for the problem $p$. The performance 
ratio is defined as \cite{ty96}:
\begin{equation}
r_{p,s}= \frac{m_{p,s}}{\min \{ m_{p,s}: s \in {\cal S}  \}}.
\label{performRatio}
\end{equation}
The performance profiles for {\tt mBFGS} and 
the Matlab {\tt fminunc} are given in Figure \ref{fig:profile1}.
Clearly, robust BFGS implementation {\tt mBFGS} finds the 
optimal solution faster than the traditional BFGS implementation
{\tt fminunc} for the tested problems.

\begin{figure}[ht]
\centerline{\epsfig{file=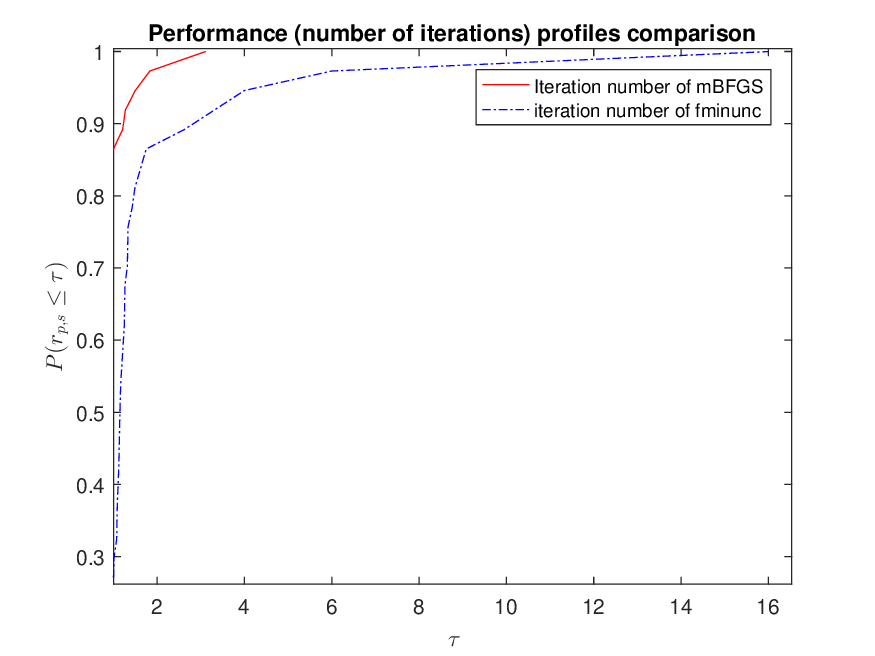,height=6cm,width=10cm}}
\caption{Performance profiles comparison for {\tt mBFGS} and 
{\tt fminunc}.}
\label{fig:profile1}
\end{figure} 

\begin{remark}
We also implemented (\ref{robustBFGS}) and 
(\ref{dirMBFGS}) and compared with the implementation of
(\ref{E1}) and (\ref{calMBFGS}) in the calculation of the 
search direction $d_k$. With the implementation of 
(\ref{robustBFGS}) and (\ref{dirMBFGS}), the 
algorithm converges with the criterion 
$\| g(x_k) \|< 10^{-5}$ as required for all the test 
problems, including {\tt brownden}, {\tt deconvu}, 
and {\tt djtl}. We noticed that although {\tt mBFGS} 
using (\ref{E1}) and (\ref{calMBFGS}) stops before 
$\| g(x_k) \|< 10^{-5}$ is achieved for these three 
problems, the objective functions obtained are essentially 
the same as if (\ref{robustBFGS}) and (\ref{dirMBFGS}) 
are used. Given the fact that using (\ref{E1}) and 
(\ref{calMBFGS}) does not require to solve the linear 
systems of equations but using (\ref{robustBFGS}) and 
(\ref{dirMBFGS}) does, we suggest using the implementation 
of (\ref{E1}) and (\ref{calMBFGS}).
\end{remark}

Most of the above problems are also used by other 
researchers \cite{hager} to test some established and 
state-of-the-art algorithms. In \cite{hager}, 
CUTEst unconstrained problems are tested against 
limited memory BFGS algorithm \cite{nocedal80} 
(implemented as {\tt L-BFGS}), a descent and conjugate 
gradient algorithm \cite{hz05} (implemented as 
{\tt CG-Descent 5.3}), and a limited memory descent 
and conjugate gradient algorithm \cite{hz12} 
(implemented as {\tt L-CG-Descent}). 
We compare the test results obtained by our implementation 
of Algorithm \ref{newAlgo} and the results obtained by 
algorithms \cite{hz12,hz05,nocedal80} (reported in 
\cite{hager}). For this test, we changed the stopping 
criterion to $\| g(x) \|_{\infty} \le 10^{-6}$
which is used in \cite{hager} for comparison between the 
results obtained in \cite{hz12,hz05,nocedal80}. The test 
results are listed in Table 2.

\footnotesize
\begin{center}
\begin{longtable}{|r|r|r|r|r|r|r|r|}
\caption{Comparison of mBFGS, L-CG-Descent, L-BFGS, and CG-Descent 5.3 for 
problems in CUTEst \cite{web}, initial points are given in CUTEst}\\
\hline    
Problem      & size   & methods          & iter   & nFun & nGrad &  obj    & gradient      \\
\hline
arglina      &  200   & mBFGS            &   1   & 14 & 2  &  1.000e+002             &   1.894e-015  \\ 
             &        & L-CG-Descent     &   1   & 3  & 2  &  {\bf 2.000e+002}       &   3.384e-008  \\
             &        & L-BFGS           &   1   & 3  & 2  &  {\bf 2.000e+002}       &   3.384e-008  \\
             &        & CG-Descent 5.3   &   1   & 3  & 2  &  {\bf 2.000e+002}       &   2.390e-007  \\
\hline
bard         &  3     & mBFGS            &   18  &  95  &  19     &   1.157e-001      &    9.765e-007           \\ 
             &        & L-CG-Descent     &   16  &  33  &  17     &   8.215e-003      &    3.673e-009           \\
             &        & L-BFGS           &   16  &  33  &  17     &   8.215e-003      &    3.673e-009           \\
             &        & CG-Descent 5.3   &   21  &  44  &  23     &   8.215e-003      &    1.912e-007           \\
\hline
beale        &  2     & mBFGS            &   13  &  76  &  14      &   4.957e-020      &    2.979e-010           \\ 
             &        & L-CG-Descent     &   15  &  31  &  16      &   2.727e-015      &    4.499e-008           \\
             &        & L-BFGS           &   15  &  31  &  16      &   2.727e-015      &    4.499e-008           \\
             &        & CG-Descent 5.3   &   18  &  37  &  19      &   1.497e-007      &    4.297e-007           \\
\hline
biggs6       &  6     & mBFGS            &   73  & 335  &  74      &   7.777e-013        &    4.920e-007         \\ 
             &        & L-CG-Descent     &   27  &  57  &  31      &   {\bf 5.656e-003}  &    2.514e-008         \\
             &        & L-BFGS           &   27  &  57  &  31      &   {\bf 5.656e-003}  &    2.514e-008         \\
             &        & CG-Descent 5.3   &   85  & 177  &  93      &   {\bf 5.656e-003}  &    9.195e-007         \\
\hline
box3         &  3     & mBFGS            &   21  &  77  &  22      &   1.692e-016        &    4.450e-008         \\ 
	     &        & L-CG-Descent     &   11  &  24  &  13      &   3.819e-013        &    7.584e-007         \\
	     &        & L-BFGS           &   11  &  24  &  13      &   3.819e-013        &    7.584e-007         \\
	     &        & CG-Descent 5.3   &   13  &  27  &  14      &   1.707e-010        &    6.003e-007         \\
\hline
brkmcc       &  2     & mBFGS            &   4   &  34  &  5     &    1.690e-001     &    8.034e-008           \\ 
             &        & L-CG-Descent     &   5   &  11  &  6     &    1.690e-001     &    6.220e-008           \\
             &        & L-BFGS           &   5   &  11  &  6     &    1.690e-001     &    6.220e-008           \\
             &        & CG-Descent 5.3   &   4   &   9  &  5     &    1.690e-001     &    5.272e-008           \\
\hline
brownbs      &  2     & mBFGS            &   632  &  13543  &  633      &   1.952e-018      &   5.369e-007     \\ 
             &        & L-CG-Descent     &   13   &  26     &   15      &   0.000e+000      &   0.000e+000     \\
             &        & L-BFGS           &   13   &  26     &   15      &   0.000e+000      &   0.000e+000     \\
             &        & CG-Descent 5.3   &   16   &  40     &   33      &   1.972e-031      &   8.882e-010     \\
\hline
brownden     &  4     & mBFGS            &   21  & 312 &   22   &   8.582e+004      &   3.092e-010            \\ 
             &        & L-CG-Descent     &   16  & 31  &   19   &   8.582e+004      &   1.282e-007            \\
             &        & L-BFGS           &   16  & 31  &   19   &   8.582e+004      &   1.282e-007            \\
             &        & CG-Descent 5.3   &   38  & 74  &   48   &   8.582e+004      &   9.083e-007            \\
\hline
chnrosnb     &  50    & mBFGS            &   160  &  2185  & 161   &  1.263e-015       &  3.525e-007             \\ 
             &        & L-CG-Descent     &   287  &   564  & 299   &  6.818e-014       &  5.414e-007             \\
             &        & L-BFGS           &   216  &   427  & 233   &  1.582e-013       &  5.565e-007             \\
             &        & CG-Descent 5.3   &   287  &   564  & 299   &  6.818e-014       &  5.414e-007             \\
\hline
cliff        &  2     & mBFGS            &   15   & 75  &  16    &   1.998e-001      &   7.602e-008            \\ 
             &        & L-CG-Descent     &   18   & 70  &  54    &   1.998e-001      &   2.316e-009            \\
             &        & L-BFGS           &   18   & 70  &  54    &   1.998e-001      &   2.316e-009            \\
             &        & CG-Descent 5.3   &   19   & 40  &  21    &   1.998e-001      &   6.352e-008            \\
\hline
cube         &  2     & mBFGS            &   21  & 134  &  22  &   4.231e-020      &   6.845e-010            \\ 
             &        & L-CG-Descent     &   32   & 77  &  47  &   1.269e-017      &   1.225e-009            \\
             &        & L-BFGS           &   32   & 77  &  47  &   1.269e-017      &   1.225e-009            \\
             &        & CG-Descent 5.3   &   33   & 80  &  49  &   6.059e-015      &   4.697e-008            \\
\hline
deconvu      &  61    & mBFGS            &   67   & 855  &  68     &   1.567e-009      &   9.999e-007            \\ 
             &        & L-CG-Descent     &   475  & 951  &  476    &   1.189e-008      &   9.187e-007            \\
             &        & L-BFGS           &   208  & 417  &  209    &   2.171e-010      &   8.924e-007            \\
             &        & CG-Descent 5.3   &   475  & 951  &  476    &   1.184e-008      &   9.078e-007            \\
\hline
denschna     &  2     & mBFGS            &   7   & 35   &  8     &   1.468e-014      &   3.198e-007            \\ 
             &        & L-CG-Descent     &   9   & 19   &  10    &   3.167e-016      &   3.527e-008            \\
             &        & L-BFGS           &   9   & 19   &  10    &   3.167e-016      &   3.527e-008            \\
             &        & CG-Descent 5.3   &   9   & 19   &  10    &   7.355e-016      &   4.825e-008            \\
\hline
denschnb     &  2     & mBFGS            &   7   & 44   &  8   &   6.048e-014      &   4.252e-007            \\ 
             &        & L-CG-Descent     &   7   & 15   &  8   &   3.641e-017      &   1.034e-008           \\
             &        & L-BFGS           &   7   & 15   &  8   &   3.641e-017      &   1.034e-008            \\
             &        & CG-Descent 5.3   &   8   & 17   &  8   &   4.702e-014      &   4.131e-007            \\
\hline
denschnc     &  2     & mBFGS            &   8   & 55   &  9     &  1.119e-021       &   1.731e-010            \\ 
             &        & L-CG-Descent     &   12  & 26   &  14    &  3.253e-019       &   3.276e-009            \\
             &        & L-BFGS           &   12  & 26   &  14    &  3.253e-019       &   3.276e-009            \\
             &        & CG-Descent 5.3   &   12  & 27   &  15    &  {\bf 1.834e-001} &   4.143e-007            \\
\hline
denschnd     &  3     & mBFGS            &   38  & 308  &  39     &  2.461e-009       &   3.146e-007            \\ 
             &        & L-CG-Descent     &   47  & 98   &  51     &  4.331e-010       &   8.483e-007            \\
             &        & L-BFGS           &   47  & 98   &  51     &  4.331e-010       &   8.483e-007            \\
             &        & CG-Descent 5.3   &   45  & 97   &  54     &  8.800e-009       &   6.115e-007            \\
\hline
denschnf     &  2     & mBFGS            &   10  &  75  &  11    &   4.325e-018      &    3.027e-008           \\ 
             &        & L-CG-Descent     &   8   &  17  &  9     &   2.126e-015      &    6.455e-007           \\
             &        & L-BFGS           &   8   &  17  &  9     &   2.126e-015      &    6.455e-007           \\
             &        & CG-Descent 5.3   &   11  &  24  &  13    &   1.104e-017      &    6.614e-008           \\
\hline
djtl         &  2     & mBFGS            &   79  &  1524  &  80    &   -8.952e+003      &    2.265e-002           \\ 
             &        & L-CG-Descent     &   82  &  917   &  880   &   -8.952e+003      &    8.865e-009           \\
             &        & L-BFGS           &   82  &  917   &  880   &   -8.952e+003      &    8.865e-009           \\
             &        & CG-Descent 5.3   &   93  &  770   &  714   &   -8.952e+003      &    3.521e-007           \\
\hline
engval2      &  3     & mBFGS            &   28  & 188  &  29   &  1.999e-018       &   9.405e-008            \\ 
             &        & L-CG-Descent     &   26  & 61   &  37   &  1.034e-016       &   8.236e-007            \\
             &        & L-BFGS           &   26  & 61   &  37   &  1.034e-016       &   8.236e-007            \\
             &        & CG-Descent 5.3   &   76  & 161  &  88   &  3.185e-014       &   5.682e-007            \\
\hline
expfit       &  2     & mBFGS            &   12  & 103  &  13    &   2.405e-001      &    2.916e-009           \\ 
     	     &        & L-CG-Descent     &   13  &  29  &  16    &   2.405e-001      &    4.208e-007           \\
             &        & L-BFGS           &   13  &  29  &  16    &   2.405e-001      &    4.208e-007           \\
             &        & CG-Descent 5.3   &   15  &  34  &  20    &   2.405e-001      &    1.758e-007           \\
\hline
growthls     &  3     & mBFGS            &   1    &  3    &  2     &   {\bf 3.542e+003} &    0.000e-999           \\ 
     	     &        & L-CG-Descent     &   143  &  425  &  299   &   1.004e+000       &    3.317e-007           \\
             &        & L-BFGS           &   143  &  425  &  399   &   1.004e+000       &    3.317e-007           \\
             &        & CG-Descent 5.3   &   441  &  997  &  596   &   1.004e+000       &    1.835e-007           \\
\hline
hairy        &  2     & mBFGS            &   19  & 160  &  20     &  2.000e+001       &   6.143e-008            \\ 
             &        & L-CG-Descent     &   36  & 99   &  65     &  2.000e+001       &   7.961e-011            \\
             &        & L-BFGS           &   36  & 99   &  65     &  2.000e+001       &   7.961e-011            \\
             &        & CG-Descent 5.3   &   14  & 35   &  24     &  2.000e+001       &   1.044e-007            \\
\hline
hatfldd      &  3     & mBFGS            &   24  & 119  &  25     &  6.615e-008       &   1.107e-007            \\ 
	     &        & L-CG-Descent     &   20  & 43   &  24     &  2.547e-007       &   1.936e-007            \\
	     &        & L-BFGS           &   20  & 43   &  24     &  2.547e-007       &   1.936e-007            \\
             &        & CG-Descent 5.3   &   40  & 98   &  61     &  6.617e-008       &   1.934e-007            \\
\hline
hatflde      &  3     & mBFGS            &   30  & 136  &  31     &  4.434e-007       &   6.576e-007            \\ 
	     &        & L-CG-Descent     &   30  & 72   &  45     &  2.000e+001       &   5.012e-007            \\
	     &        & L-BFGS           &   30  & 72   &  45     &  2.000e+001       &   5.012e-007            \\
             &        & CG-Descent 5.3   &   53  & 120  &  72     &  2.000e+001       &   5.012e-007            \\
\hline
heart6ls     &  6     & mBFGS            &   2266  & 8700  & 2267    &  2.865e-023       &   6.934e-009        \\ 
             &        & L-CG-Descent     &   684   & 1576  &  941    &  2.646e-010       &   5.562e-007        \\
             &        & L-BFGS           &   684   & 1576  &  941    &  2.646e-010       &   5.562e-007        \\
             &        & CG-Descent 5.3   &   2570  & 5841  & 3484    &  1.305e-010       &   2.421e-007        \\
\hline
helix        &  3     & mBFGS            &   22  & 154  &  23    &   5.489e-017      &    1.349e-007           \\ 
             &        & L-CG-Descent     &   23  & 49   &  27    &   1.604e-015      &    3.135e-007           \\
             &        & L-BFGS           &   23  & 49   &  27    &   1.604e-015      &    3.135e-007           \\
             &        & CG-Descent 5.3   &   44  & 90   &  46    &   2.427e-013      &    6.444e-007           \\
\hline 
himmelbb     &  2     & mBFGS            &   1  &  23  &   2    &   9.665e-014      &    8.167e-008           \\ 
             &        & L-CG-Descent     &   10 &  28  &  21    &   9.294e-013      &    2.375e-007           \\
             &        & L-BFGS           &   10 &  28  &  21    &   9.294e-013      &    2.375e-007           \\
             &        & CG-Descent 5.3   &   11 &  23  &  12    &   1.584e-013      &    1.084e-008           \\
\hline
himmelbg     &  2     & mBFGS            &   7  &  45  &   8    &   1.070e-013      &    9.071e-007           \\ 
             &        & L-CG-Descent     &   8  &  20  &  13    &   9.294e-013      &    2.375e-007           \\
             &        & L-BFGS           &   8  &  20  &  13    &   9.294e-013      &    2.375e-007           \\
             &        & CG-Descent 5.3   &   10 &  24  &  15    &   1.584e-013      &    1.084e-008           \\
\hline
himmelbh     &  2     & mBFGS            &   7  &  45  &   8    &    -1.000e+000     &    5.026e-007           \\ 
             &        & L-CG-Descent     &   7  &  16  &   9    &    -1.000e+000     &    2.892e-011           \\
             &        & L-BFGS           &   7  &  16  &   9    &    -1.000e+000     &    2.892e-011           \\
             &        & CG-Descent 5.3   &   7  &  16  &   9    &    -1.000e+000     &    1.381e-007           \\
\hline
humps        &  2     & mBFGS            &   104  & 857  & 105     &    3.280e-016     &    6.351e-009           \\ 
             &        & L-CG-Descent     &   53   & 165  & 120     &    3.682e-012     &    8.552e-007           \\
             &        & L-BFGS           &   53   & 165  & 120     &    3.682e-012     &    8.552e-007           \\
             &        & CG-Descent 5.3   &   48   & 140  & 101     &    3.916e-012     &    8.774e-007           \\
\hline
jensmp       &  2     & mBFGS            &   1  &   3  &   2    &  {\bf 2.020e+003} &    0.000e-999           \\ 
             &        & L-CG-Descent     &   15 &  33  &  22    &    1.244e+002     &    5.302e-010           \\
             &        & L-BFGS           &   15 &  33  &  22    &    1.244e+002     &    5.302e-010           \\
             &        & CG-Descent 5.3   &   13 &  29  &  19    &    1.244e+002     &    4.206e-009           \\
\hline
kowosb       &  4     & mBFGS            &   28 &  147  &  29    &    3.075e-004     &    1.367e-007           \\ 
             &        & L-CG-Descent     &   17 &   39  &  23    &    3.078e-004     &    3.704e-007           \\
             &        & L-BFGS           &   17 &   39  &  23    &    3.078e-004     &    3.704e-007           \\
             &        & CG-Descent 5.3   &   66 &  139  &  76    &    3.078e-004     &    8.818e-007           \\
\hline
loghairy     &  2     & mBFGS            &   74  & 882  &  75   &    1.823e-001     &    5.904e-007           \\ 
             &        & L-CG-Descent     &   27  & 81   &  58   &    1.823e-001     &    1.762e-007           \\
             &        & L-BFGS           &   27  & 81   &  58   &    1.823e-001     &    1.762e-007           \\
             &        & CG-Descent 5.3   &   46  & 136  &  97   &    1.823e-001     &    7.562e-008           \\
\hline
mancino      &  100   & mBFGS            &   37   &  1202   &   38  &     1.548e-020    &     1.414e-007          \\ 
             &        & L-CG-Descent     &   11   &   23    &   12  &     9.245e-021    &     7.239e-008          \\
             &        & L-BFGS           &   9    &   19    &   30  &     3.048e-021    &     1.576e-007          \\
             &        & CG-Descent 5.3   &   11   &   23    &   12  &     9.245e-021    &     7.239e-008          \\
\hline
maratosb     &  2     & mBFGS            &   3    &  59     &   4     &   -1.000e+000      &   5.142e-008            \\ 
             &        & L-CG-Descent     &   1145 &  3657   &  2779   &   -1.000e+000      &   3.216e-007            \\
             &        & L-BFGS           &   1145 &  3657   &  2779   &   -1.000e+000      &   3.216e-007            \\
             &        & CG-Descent 5.3   &   946  &  2911   &  2191   &   -1.000e+000      &   3.230e-009            \\
\hline
mexhat       &  2     & mBFGS            &   5   &  32  &   6      &    -4.010e-002     &     1.426e-012          \\ 
             &        & L-CG-Descent     &   20  &  56  &   39     &    -4.001e-002     &     4.934e-009          \\
             &        & L-BFGS           &   20  &  56  &   39     &    -4.001e-002     &     4.934e-009          \\
             &        & CG-Descent 5.3   &   27  &  61  &   36     &    -4.001e-002     &     3.014e-007          \\
\hline
osborneb     &  11    & mBFGS            &   53  &  377  &  54     &    4.014e-002     &     2.480e-007          \\ 
	     &        & L-CG-Descent     &   62  &  127  &  65     &    4.014e-002     &     4.427e-007          \\
	     &        & L-BFGS           &   62  &  127  &  65     &    4.014e-002     &     4.427e-007          \\
             &        & CG-Descent 5.3   &  214  &  423  & 219     &    4.014e-002     &     7.485e-007          \\
\hline
palmer1c     &  8     & mBFGS            &   32     &  211   &  33     &    9.760e-002     &     3.935e-007          \\ 
             &        & L-CG-Descent     &   11     &   26   &  26     &    9.761e-002     &     1.254e-009          \\
             &        & L-BFGS           &   11     &   26   &  26     &    9.761e-002     &     1.254e-009          \\
             &        & CG-Descent 5.3   &   126827 & 224532 & 378489  &    9.761e-002     &     9.545e-007          \\
\hline
palmer2c     &  8     & mBFGS            &   112    &  446   &  113     &   1.442e-002      &     8.296e-007          \\ 
             &        & L-CG-Descent     &   11     &  21    &  21      &   1.437e-002      &     1.257e-008          \\
             &        & L-BFGS           &   11     &  21    &  21      &   1.437e-002      &     1.257e-008          \\
             &        & CG-Descent 5.3   &   21362  & 21455  & 42837    &   1.437e-002      &     5.761e-007          \\
\hline
palmer3c     &  8     & mBFGS            &   47   & 245  &  48   &   1.954e-002      &     2.050e-008          \\ 
             &        & L-CG-Descent     &   11   &  20  &  20   &   1.954e-002      &     1.754e-010          \\
             &        & L-BFGS           &   11   &  20  &  20   &   1.954e-002      &     1.754e-010          \\
             &        & CG-Descent 5.3   &   5536 & 5777 & 11379 &   1.954e-002      &     9.753e-007          \\
\hline
palmer4c     &  8     & mBFGS            &   78    & 351    &  79    &   5.031e-002      &     2.235e-007          \\ 
             &        & L-CG-Descent     &   11    &  20    &  20    &   5.031e-002      &     3.928e-009          \\
             &        & L-BFGS           &   11    &  20    &  20    &   5.031e-002      &     3.928e-009          \\
             &        & CG-Descent 5.3   &   44211 & 49913  & 96429  &   5.031e-002      &     9.657e-007          \\
\hline
palmer5c     &  6     & mBFGS            &   13  &  157  &  14   &    2.128e+000     &      4.810e-009         \\ 
             &        & L-CG-Descent     &   6   &  13   &  7    &    2.128e+000     &      3.749e-012         \\
             &        & L-BFGS           &   6   &  13   &  7    &    2.128e+000     &      3.749e-012         \\
             &        & CG-Descent 5.3   &   6   &  13   &  7    &    2.128e+000     &      2.629e-009         \\
\hline
palmer6c     &  8     & mBFGS            &   56    &  243   &   57   &   1.639e-002      &     6.900e-007          \\ 
             &        & L-CG-Descent     &   11    &  24    &   24   &   1.639e-002      &     5.520e-009          \\
             &        & L-BFGS           &   11    &  24    &   24   &   1.639e-002      &     5.520e-009          \\
             &        & CG-Descent 5.3   &   14174 & 142228 & 28411  &   1.639e-002      &     7.738e-007          \\
\hline
palmer7c     &  8     & mBFGS            &   41    &  212  &   42     &   6.020e-001      &     5.201e-007          \\ 
             &        & L-CG-Descent     &   11    &   20  &   20     &   6.020e-001      &     7.132e-009          \\
             &        & L-BFGS           &   11    &   20  &   20     &   6.020e-001      &     7.132e-009          \\
             &        & CG-Descent 5.3   &   65294 & 78428 & 149585   &   6.020e-001      &     9.957e-007          \\
\hline
palmer8c     &  8     & mBFGS            &   48   &   361  &  49    &   1.598e-001      &     1.099e-009          \\ 
             &        & L-CG-Descent     &   11   &   18   &  17    &   1.598e-001      &     2.376e-009          \\
             &        & L-BFGS           &   11   &   18   &  17    &   1.598e-001      &     2.376e-009          \\
             &        & CG-Descent 5.3   &   8935 & 9903   & 19183  &   1.598e-001      &     9.394e-007          \\
\hline
rosenbr      &  2     & mBFGS            &   32  &  241  &  33   &   1.383e-016      &     4.603e-007          \\ 
             &        & L-CG-Descent     &   34  &   77  &  44   &   4.691e-018      &     7.167e-008          \\
             &        & L-BFGS           &   34  &   77  &  44   &   4.691e-018      &     7.167e-008          \\
             &        & CG-Descent 5.3   &   37  &   86  &  52   &   1.004e-014      &     1.894e-007          \\
\hline
sineval      &  2     & mBFGS            &   69  &  489  &  70   &   1.910e-019      &     1.168e-008          \\ 
             &        & L-CG-Descent     &   60  &  143  &  87   &   1.556e-023      &     1.817e-011          \\
             &        & L-BFGS           &   60  &  143  &  87   &   1.556e-023      &     1.817e-011          \\
             &        & CG-Descent 5.3   &   62  &  172  & 122   &   1.023e-012      &     5.575e-007          \\
\hline
sisser       &  2     & mBFGS            &   19  &  83  &  20   &    3.860e-010     &     4.587e-007          \\ 
             &        & L-CG-Descent     &   6   &  18  &  14   &    6.830e-012     &     2.220e-008          \\
             &        & L-BFGS           &   6   &  18  &  14   &    6.830e-012     &     2.220e-008          \\
             &        & CG-Descent 5.3   &   6   &  13  &   7   &    3.026e-014     &     3.663e-010          \\
\hline
tointqor     &  50    & mBFGS            &   39  &  615  &  40   &   1.176e+003      &     4.033e-007          \\ 
             &        & L-CG-Descent     &   29  &  36   &  53   &   1.175e+003      &     4.467e-007          \\
             &        & L-BFGS           &   28  &  35   &  51   &   1.175e+003      &     7.482e-007          \\
             &        & CG-Descent 5.3   &   29  &  36   &  53   &   1.175e+003      &     4.464e-007          \\
\hline
vardim       &  200   & mBFGS            &   22  & 154  &  23     &   1.237e-021      &     1.376e-009          \\ 
             &        & L-CG-Descent     &   10  &  21  &  11     &   4.168e-019      &     2.582e-007          \\
             &        & L-BFGS           &   7   &  31  &  27     &   5.890e-025      &     3.070e-010          \\
             &        & CG-Descent 5.3   &   10  &  21  &  11     &   4.168e-019      &     2.582e-007          \\
\hline
watson       &  12    & mBFGS            &   61  &  308  &  62    &   1.130e-008      &     3.081e-007          \\ 
             &        & L-CG-Descent     &   49  &  102  &  54    &   1.592e-007      &     8.026e-007          \\
             &        & L-BFGS           &   48  &  97   &  49    &   9.340e-008      &     1.319e-007          \\
             &        & CG-Descent 5.3   &   726 &  145  &  727   &   1.139e-007      &     8.115e-007          \\
\hline
yfitu        &  2     & mBFGS            &   73  & 462   &  74     &   6.670e-013      &     1.938e-007          \\ 
             &        & L-CG-Descent     &   75  & 177   &  106    &   8.074e-010      &     3.910e-007          \\
             &        & L-BFGS           &   75  & 177   &  106    &   8.074e-010      &     3.910e-007          \\
             &        & CG-Descent 5.3   &   147 & 327   &  189    &   2.969e-011      &     5.681e-007          \\
\hline
\end{longtable}
\end{center}
\normalsize

We summarize the comparison of the test results as follows:
\begin{itemize}
\item[1.] For two problems ({\tt arglina} and {\tt biggs6}), 
{\tt mBFGS} converges to better points
than {\tt L-CG-Descent}, {\tt L-BFGS}, and {\tt CG-Descent 5.3}. 
For another 2 problems ({\tt growthls} 
and {\tt jensmp}), {\tt L-CG-Descent}, {\tt L-BFGS}, and 
{\tt CG-Descent 5.3} converge to better points. 
\item[2.] For 19 problems, {\tt mBFGS} converges faster 
than {\tt L-CG-Descent}, {\tt L-BFGS}, and 
{\tt CG-Descent 5.3}. For about 10 problems, {\tt mBFGS} 
converges slower than {\tt L-CG-Descent}, {\tt L-BFGS}, and 
{\tt CG-Descent 5.3}. For the rest problems, {\tt mBFGS} 
converges either faster than some but slower than other
codes or as faster as all other codes.
\end{itemize}

The performance profiles for {\tt mBFGS}, {\tt L-CG-Descent},
{\tt L-BFGS}, and {\tt CG-Descent 5.3} are given in Figure 
\ref{fig:profile2}. Clearly, robust BFGS implementation {\tt mBFGS} 
finds the optimal solution faster than other modified BFGS
algorithms for the tested problems. The performances of 
{\tt L-CG-Descent} and {\tt L-BFGS} are similar. The performance
of {\tt CG-Descent 5.3} is not as good as {\tt L-CG-Descent} and 
{\tt L-BFGS}.

\begin{figure}[ht]
\centerline{\epsfig{file=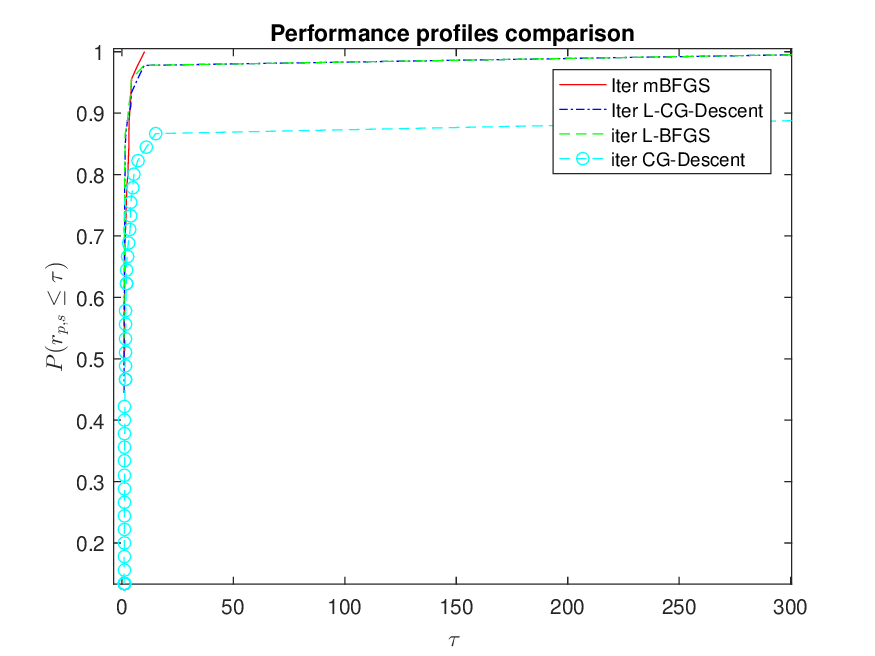,height=6cm,width=10cm}}
\caption{Performance profiles comparison for {\tt mBFGS}, 
{\tt L-CG-Descent}, {\tt L-BFGS}, and {\tt CG-Descent 5.3}.}
\label{fig:profile2}
\end{figure} 

Based on these numerical test results, we believe that the proposed 
algorithm is very promising.

\section{Conclusions}
We have proposed a robust BFGS algorithm that converges 
to a local optimum under some mild assumptions
for not only convex optimization problems but also for
non-convex optimization problems. In addition, we proved that 
the robust BFGS algorithm is globally and superlinearly 
convergent for the worst case, the non-convex optimization 
problems. We have shown that the computational cost in
each iteration is almost the same for both the BFGS algorithm
and the robust BFGS algorithm. We have provided numerical 
test results and compared the performance of the robust 
BFGS to the performance of other established and 
state-of-the-art algorithms, such as BFGS, limited memory 
BFGS, descent and conjugate gradient, and limited 
memory descent and conjugate gradient. The results and
comparison show that the robust BFGS algorithm appears 
very efficient and effective.

\section{Acknowledgments}
The author is indebted to the anonymous reviewer and the
associate editor for their constructive comments and suggestions
that are invaluable in the preparation of the final version of the paper. 
The author would like to thank Mr. Mike Case, the Director 
of the Division of Engineering in the Office of 
Research at US NRC, and Dr. Chris Hoxie, in the Office 
of Research at US NRC, for their providing computational 
environment for this research.
The author is grateful to Professor Teresa Monterio at 
Universidade do Minho for providing the software to 
convert AMPL mod-files into nl-files, which makes the 
test possible. 

\section{Conflict of interest}

On behalf of all authors, the corresponding author states that there is no conflict of
interest.



\end{document}